\documentclass[a4paper,11pt]{article}
\usepackage[applemac]{inputenc}
\usepackage{sectsty}

\makeatother
\usepackage[T1]{fontenc}
\usepackage{lmodern}

\usepackage{amsthm,amsmath, amsfonts, amssymb}
\usepackage{graphicx}
\usepackage{float}
\usepackage{indentfirst}
\usepackage{epstopdf}
\usepackage{array}
\usepackage{amssymb}
\usepackage{amsmath}
\theoremstyle{definition}

\newtheorem{cor}{Corollary}
\newtheorem{prop}{Proposition}
\newtheorem{rem}{Remark}
\newtheorem{defi}{Definition}
\newtheorem{theo}{Theorem}

\newcommand {\R}{\mathbb R}
\newcommand {\C}{\mathbb C}

\def\tribare#1{\left|\!\left|\!\left|{#1}\right|\!\right|\!\hskip-0.2pt\right|}

\begin{document} 
\author{Antoine Derighetti}
\title{On the mutipliers of the Fig$\grave{\hbox{a}}$\hskip2pt-Talamanca  Herz algebra}\maketitle

{\bf Abstract.} Let $G$ be a locally compact group and $p,q \in \R$ with $p>1, \hskip2pt p\not =2$ and $q$ between $2$ and $p$ (if $p<2$ then $p<q<2,$ if $p>2$ then $2<q<p$.) The main result of the paper is that $ A_{q}(G)$ multiplies $A_p(G),$ more precisely we show that the Banach algebra $A_p(G)$ is a Banach module on $A_q(G).$ 
\vskip10pt
\section{Introduction}
Let $G$ be a locally compact group, in [1] an entirely selfcontained proof of the following result was given (section 8.3 Theorem 8 p. 158): for every $1<p<\infty$  $A_p(G)$ is a Banach module over $A_2(G)$: for every $u\in A_2(G)$ and for every $v\in A_p(G)$ we have $uv\in A_p(G)$ and 
$\|uv\|_{A_p} \leq \|u\|_{A_2} \|v\|_{A_p}.$ Let now $p, q \in \R$ with $p>1$, $p\not =2$ and $q$
 between $2$ and $p.$ We show (Theorem 8) that for every $u\in A_q(G)$ and for every $v\in A_p(G)$ we have $uv\in A_p(G)$ and $\|uv\|_{A_p} \leq \|u\|_{A_q} \|v\|_{A_p}.$
 \vskip1pt
We recall some definitions. For $\varphi$ a map of $G$ into $\C$ we set $\check \varphi(x)=\varphi(x^{-1}),\hskip3pt _{a}\varphi(x)=\varphi(ax)$ and $\varphi_{a}(x)=\varphi(xa)$ for $a, x\in G.$ For $1< p<\infty$  $L^{p}(G)$ is the $L^p$-space with respect to a left Haar measure of $G.$ Let $\mathcal L(L^{p}(G))$  be the  Banach algebra of all linear bounded operators of $L^{p}(G)$, the operator norm of an operator $T$ is denoted $\tribare{T}_p$. An element $T$ of $\mathcal L(L^{p}(G))$ is said to be  a $p$\hskip1pt-convolution operator (written $T\in CV_p(G)$) if $T(_{a} \varphi)=\hskip1pt_{a} T(\varphi)$ for every $a\in G$ and for every $\varphi \in L^{p}(G).$
\vskip1pt
We denote by $C_{00}(G)$ the set of all continuous  complex valued functions on $G$ with compact support. Let $\mu$ be a Radon measure on $G$ we put $\check{\mu}(\varphi)=\mu(\check{\varphi}).$ Suppose that $\mu$ is a bounded  Radon measure on $G$ ($\mu \in M^{1}(G)$), there is a unique $S\in \mathcal L(L^{p}(G))$ such that
$$S[\varphi]=\big[\varphi *\big(\Delta_{G}^{1/{p'}}\check{\mu}\big)\big]$$for every $\varphi\in C_{00}(G).$ The operator $S$ is denoted $\lambda_{G}^{p}(\mu).$ We have $\lambda_{G}^{p}(\mu) \in CV_p(G)$ and $\tribare{\lambda_{G}^{p}(\mu) }_p \leq \|\mu\|.$ We also have $\lambda_{G}^{p}(\delta_{a}) \varphi=\varphi_a \Delta_{G}(a)^{1/p}.$ In [1] we proved that every $1<p<\infty$
and for every amenable locally compact group $G$ the following deep inequality holds: 
$\tribare{\lambda_{G}^{2}(\mu)}_2\leq \tribare{\lambda_{G}^{p}(\mu)}_p$ for every $\mu\in M^{1}(G).$ Even for finite abelian groups this result is not trivial.

In this paper we show (Corollary 10) that for every amenable locally compact group $G$, for $p,q \in \R$ with $p>1, \hskip2pt p\not =2$ and $q$ between $2$ and $p$ we have $\tribare{\lambda_{G}^{q}(\mu)}_q\leq \tribare{\lambda_{G}^{p}(\mu)}_p$ for every $\mu\in M^{1}(G).$

Let $G$ be a locally compact group and $1<p<\infty.$ We denote by $\mathcal A_p(G)$ the set of all pairs $((k_n)_{n=1}^{\infty},(l_n)_{n=1}^{\infty})$ where $(k_n)_{n=1}^{\infty}$ is a sequence of $\mathcal L^{p}(G)$ and $(l_n)_{n=1}^{\infty}$ is a sequence of $\mathcal L^{p'}(G)$ with
$$\sum_{n=1}^{\infty}N_p(k_n) N_{p'}(l_n) <\infty.$$The topology on $\mathcal L(L^p(G)),$ associated to the family of seminorms
$$T\mapsto \Bigg|\sum_{n=1}^{\infty} <T[k_n],[l_n]>\Bigg|$$with $((k_n)_{n=1}^{\infty},(l_n)_{n=1}^{\infty})\in \mathcal A_p(G),$ is called the ultraweak topology. The ultraweak closure of $\lambda_{G}^{p}(M^1(G))$ in $\mathcal L(L^p(G))$ is denoted $PM_p(G).$ Every element of $PM_p(G)$ is called a $p$\hskip1pt-pseudomeasure. The space $PM_p(G)$ can be identified as the dual of a Banach algebra $A_p(G)$ of continuous continuous functions on $G.$

We denote by $A_p(G)$ the set:
$$\Big\{u: G \rightarrow \C :  \text{there is} \hskip3pt ((k_n)_{n=1}^{\infty},(l_n)_{n=1}^{\infty}))\in \mathcal A_{p}(G)\hskip3pt  \text{such that}$$
$$\hskip3pt u(x)=\sum_{n=1}^{\infty} \big(\overline{k_n} *\hskip1pt\check{l_n} \big)(x)\hskip3pt\text{for every}\hskip3pt x\in G \Big\}.$$For $u\in A_p(G)$ we put
$$\|u\|_{A_p}:=\text{inf}\Bigg\{\sum_{n=1}^{\infty} N_p (k_n) N_{p'}(l_n) : ((k_n)_{n=1}^{\infty},(k_n)_{n=1}^{\infty}) \in \mathcal A_p(G) $$$$\hskip3pt\text{such that}\hskip3pt u=\sum_{n=1}^{\infty}\overline{k_n} *\hskip1pt\check{l_n} \Bigg\}.$$For the pointwise product $A_p(G)$ is a Banach algebra, called the Figà-Talamanca Herz of $G.$ We have $A_p(G)\subset C_{0}(G) $ and for every $u\in A_p(G)$ the following inequality holds:  $\|u\|_{\infty}\leq \|u\|_{A_p}.$ Moreover the space $A_p(G)\cap C_{00}(G)$ is dense in $A_p(G).$

For every   $T\in PM_p(G),$ we put 
$$\Psi_{G}^{p}(T)(u)=\sum_{n=1}^{\infty}\overline{<T[\tau_p k_n],[\tau_{p'}l_n]>}$$for every $u\in A_p(G)$ and for every $((k_n)_{n=1}^{\infty},(k_n)_{n=1}^{\infty})\in \mathcal A_p(G)$ such that
$$u=\sum_{n=1}^{\infty} \overline{k_n}*\check{l_n}$$where $\tau_p \varphi(x)=\varphi(x^{-1}) \Delta_{G}(x^{-1})^{1/p}.$
Then $\Psi_{G}^{p}$ is a conjugate linear isometry of $PM_p(G)$ onto $A_p(G)'$ with the following properties:
\vskip1pt
1.  $\Psi_{G}^{p}(\lambda_{G}^{p}(\tilde{\mu}))= \mu$ for every $\mu \in M^{1}{G},$ where $\tilde{\varphi}(x)= \overline{\varphi(x^{-1})}$ and\vskip3pt $\tilde{\mu}(\varphi)=\overline{\mu(\tilde{\varphi})}.$
 \vskip1pt
 2. $\Psi_{G}^{p}$ is an homeomorphism of $PM_p(G)$, with the ultraweak topology, onto $A_p(G)',$ with the weak topology $\sigma({A_p}', A_p).$
\vskip1pt
We will use the fact (see [1] Chapter 5) that $CV_p(G)$ carries a natural structure of left normed $A_p(G)$-module $(u, T)\mapsto uT.$ We shortly  recall the definition of $uT$: for $k\in \mathcal L^{p}(G),$ $l\in \mathcal L^{p'}(G)$ and $T\in CV_p(G)$ there is a unique linear bounded operator of $L^{p}(G),$ denoted $(\overline{k}*\check{l})T,$ such that
$$\big<\big((\overline{k}*\check{l})T\big)[\varphi],[\psi]\big>
=\int_{G} \big<T[_{t^{-1}}(\check{k}) \varphi],[_{t^{-1}}(\check{l}) \psi]\big> dt$$for every $\varphi, \psi \in C_{00}(G).$ Then, for $u\in A_p(G),$ we put $uT=\displaystyle\sum_{n=1}^{\infty}(\overline{k_n}*\check{l_n})T$ for every $((k_n)_{n=1}^{\infty},(l_n)_{n=1}^{\infty})\in \mathcal A_p(G)$ such that $u=\displaystyle\sum_{n=1}^{\infty} \overline{k_n}*\check{l_n}.$
\section{The general case}
In this paragraph $G$ is a general locally compact group, $G$ is not assumed to be amenable.
\begin{theo} Let $G$ be a locally compact group, $1<p<\infty, \hskip1pt u\in A_p(G),  \hbox{$T \in CV_p(G)$}$ and $\varphi\in L^p(G) \cap L^2(G).$ 

1. $\tau_{p}(uT)\tau_p \varphi \in L^2(G),$

2. $ \tribare{\tau_{p}(uT)\tau_p \varphi}_2 \leq \|u\|_{A_p} \tribare{T}_p \|\varphi\|_2$\hskip2pt.
\end{theo}
\begin{proof}  See [1] section 8.3 Theorem 1 p. 156.
\end{proof}
\setcounter{prop}{1}
\begin{prop} Let $G $ be a locally compact group, $p, q\in \R$ with $p>1,  p\not=2$ and $q$ between $p$ and $2$. We denote by $\mathcal S$ the set
$\{[r] : \hbox{$ r\in \mathcal L^1(G)$}, r\hskip3pt \text{step function}\}.$ Then for every $u\in A_p(G)$
for every $T\in CV_p(G)$ and every $\varphi \in \mathcal S$ we have: for every $1<p<\infty$

1. $\tau_{p}(uT)\tau_p \varphi \in L^q(G),$

2. $ \tribare{\tau_{p}(uT)\tau_p \varphi}_q \leq \|u\|_{A_p} \tribare{T}_p \|\varphi\|_q$\hskip2pt.
\end{prop}
\begin{proof}  
(I) We prove Proposition 2 if $p<2.$

We put$$t=\frac{2(q-p)}{q(2-p)}.$$We have $0<t<1$ and
$$\frac{1}{q}=\frac{1-t}{p} +\frac{t}{2}.$$According to Riesz-Thorin and to Theorem 1\hskip5pt $\tau_{p}(uT)\tau_p \varphi \in L^q(G)$ and
$$ \tribare{\tau_{p}(uT)\tau_p \varphi}_q \leq \tribare{uT}^{1-t}_p \big(\|u\|_{A_p} \tribare{T}_p\big)^{t}\|\varphi\|_q$$but
$$\tribare{uT}^{1-t}_p \big(\|u\|_{A_p} \tribare{T}_p\big)^{t}\|\varphi\|_q
\leq \|u\|_{A_p} ^{1-t} \tribare{T}_p^{1-t}  \|u\|_{A_p}^{t} \tribare{T}_p^{t} \|\varphi\|_q$$
$$= \|u\|_{A_p} \tribare{T}_p \|\varphi\|_q\hskip2pt.$$

(II) We prove Proposition 2 if $p>2.$

We put$$t=\frac{(2-q)p}{(2-p)q}.$$We have $0<t<1$ and
$$\frac{1}{q}= \frac{1-t}{2} +\frac{t}{p}.$$As in (I)  $\tau_{p}(uT)\tau_p \varphi \in L^q(G)$ and
$$ \tribare{\tau_{p}(uT)\tau_p \varphi}_q \leq  \big(\|u\|_{A_p} \tribare{T}_p\big)^{1-t}  \tribare{uT}_p^{t}\|\varphi\|_q $$and consequently  $ \tribare{\tau_{p}(uT)\tau_p \varphi}_q \leq \|u\|_{A_p} \tribare{T}_p \|\varphi\|_q$\hskip2pt.
\end{proof}
\setcounter{theo}{2}
\begin{theo}Let $G $ be a locally compact group, $p, q\in \R$ with $p>1, \hskip2pt p\not=2$ and $q$ between $p$ and $2$. Then for every $u\in A_p(G),$ for every $T\in CV_p(G)$ and every $\varphi\in L^p(G)\cap L^q(G)$ we have:

1. $\tau_{p}(uT)\tau_p \varphi \in L^q(G),$

2. $ \tribare{\tau_{p}(uT)\tau_p \varphi}_q \leq \|u\|_{A_p} \tribare{T}_p \|\varphi\|_q$\hskip2pt.
\end{theo}
\begin{defi} Let $G $ be a locally compact group, $p, q\in \R$ with $p>1, \hskip2pt p\not=2$ and $q$ between $p$ and $2$. We denote by $E_{p,q}$ the set of all $T\in CV_p(G)$ such that :

1. $\tau_p T \tau_p \varphi \in L^q(G)$ for every $\varphi \in L^p(G) \cap L^q(G) ,$

2. There is $C>0$ such that $ \tribare{\tau_{p} T \tau_p \varphi}_q \leq C\|\varphi\|_q$ for every $\varphi \in L^p(G) \cap L^q(G) .$

\end{defi}
\setcounter{prop}{3}
\begin{prop}Let $G $ be a locally compact group, $p, q\in \R$ with $p>1, \hskip1pt p\not=2$ and $q$ between $p$ and $2$. The set $E_{p,q}$ is a subalgebra of $CV_p(G).$
\end{prop}
\begin{proof} Let $S, T\in E_{p,q}$ , $C, C'$
 the constants of Definition 1 and $\varphi\in L^p(G) \cap L^q(G.$ From $\tau_p T \tau_p \varphi \in L^p(G) \cap L^q(G)$ it follows $ (\tau_p S \tau_p)( \tau_p T \tau_p \varphi) \in L^p(G) \cap L^q(G)$
and consequently  $\tau_p ST \tau_p \varphi \in L^p(G) \cap L^q(G).$ Moreover
$$\|\tau_p ST \tau_p \varphi\|_q=\|\tau_p S \tau_p(\tau_p T \tau_p \varphi)\|_{q} \leq C \|\tau_p T \tau_p \varphi\|_q \leq C C' \|\varphi\|_q$$this implies that $ST\in E_{p,q}\hskip2pt.$
\end{proof}
\begin{prop}Let $G $ be a locally compact group, $p, q\in \R$ with $p>1, \hskip1pt p\not=2$ and $q$ between $p$ and $2$. For every $\mu \in M^1(G)$ we have $\lambda_{G}^p(\mu)\in E_{p,q}$ and for every $\varphi\in L^p(G) \cap L^q(G)$ $\tau_p \lambda_{G}^p (\mu) \tau_p \varphi=\rho_{G}^q (\mu) \varphi.$

\end{prop}
\begin{proof} For every $\varphi \in C_{00}(G)$  we have $\tau_{p}\big(\tau_p\varphi *\Delta_{G} ^{1/{p'}}\check \mu\big)=\varphi *\mu.$ Let $\varphi \in L^p(G) \cap L^q(G)$ we obtain
$\tau_p \lambda_{G}^p (\mu) \tau_p \varphi=\rho_{G}^q (\mu) \varphi.$ But $\rho_{G}^q (\mu) \varphi$ belongs to $L^q(G)$ and therefore $\tau_p \lambda_{G}^p (\mu) \tau_p \varphi \in L^q(G).$
Finally the inequality $\| \rho_{G}^q (\mu) \varphi\|_q \leq \|\mu\| \|\varphi\|_q$ implies $\|\tau_p \lambda_{G}^p (\mu) \tau_p \varphi \|_q \leq \|\mu\| \|\varphi\|_q.$

\end{proof}
\begin{prop}Let $G $ be a locally compact group, $p, q\in \R$ with $p>1, \hskip1pt p\not=2$ and $q$ between $p$ and $2$. For every $T\in E_{p,q}$ there is a unique $S\in \mathcal L(L^q(G))$ such that $S\varphi= \tau_pT\tau_p \varphi$ for every $\varphi\in L^p(G) \cap L^q(G).$ We have $\tau_q S \tau_q \in CV_q(G).$

\end{prop}
\begin{defi}Let $G $ be a locally compact group, $p, q\in \R$ with $p>1, \hskip1pt p\not=2$ and $q$ between $p$ and $2$, $T\in E_{p,q}$ and $S$ as in Proposition 6. We put $\alpha_{p,q}(T)= \tau_q S \tau_q.$
\end{defi}
\setcounter{theo}{6}
\begin{theo}Let $G $ be a locally compact group, $p, q\in \R$ with $p>1, \hskip1pt p\not=2$ and $q$ between $p$ and $2$. Then:

1. $\alpha_{p,q}$ is a monomorphism of the algebra $E_{p,q}$ into $CV_q(G),$

2.  $\alpha_{p,q}(\lambda_{G}^p(\mu))=\lambda_{G}^q(\mu)$ for every $\mu\in M^1(G),$

3. for every $u\in A_p(G)$ and for every $T\in CV_p(G)$ we have $uT \in E_{p,q}$ and $\tribare{\alpha_{p,q}(uT)}_q \leq \|u\|_{A_p} \tribare{T}_p,$

4. for every $T\in PM_p(G) \cap E_{p,q}$ we have $\alpha_{p,q}(T)\in PM_q(G).$

\end{theo}
\begin{proof}

(I) We prove the assertion 1.

Let $T, T'\in E_{p,q}$  $C, C'>0$\hskip5pt $S, S' \in \mathcal L(L^q(G))$ such that $\|\tau_p T \tau_p \varphi\|_q \leq C\|\varphi\|_q$, $\|\tau_p T' \tau_p \varphi\|_q \leq C'\|\varphi\|_q$, $S\varphi =\tau_pT \tau_p \varphi,$ and $S'\varphi =\tau_pT' \tau_p \varphi$ for every $\varphi \in L^p(G)\cap L^q(G).$ We have $\alpha_{p,q}(T)=\tau_qS\tau_q$  and $\alpha_{p,q}(T')=\tau_qS'\tau_q.$

$(\text{I})_{1}$  $\alpha_{p,q}(T+T')= \alpha_{p,q}(T)+ \alpha_{p,q}(T').$

There is $S''\in \mathcal L(L^q(G))$ such that $\tau_p(T+T') \tau_p \varphi=S''\varphi$ for every $\varphi \in L^p(G)\cap L^q(G).$ We have $\alpha_{p,q}(T+T')=\tau_qS''\tau_q.$ For $ \varphi \in L^p(G)\cap L^q(G)$ we get $\tau_p(T+T')\tau_p \varphi=\tau_p(T\tau_p\varphi +T'\tau_p\varphi)=\tau_pT\tau_p\varphi+\tau_pT'\tau_p\varphi=S\varphi+S'\varphi=S''\varphi.$ Consequently $S+S'=S''.$Thus $\alpha_{p,q}(T)+\alpha_{p,q}(T')= \alpha_{p,q}(T+T').$

$(\text{I})_{2}$ $\alpha_{p,q} (\gamma T)=\gamma \hskip2pt\alpha_{p,q}(T)$ for $\gamma \in \C$ and $\alpha_{p,q}(TT')= \alpha_{p,q}(T) \alpha_{p,q}(T').$

The proof of $(\text{I})_{2}$ is similar to the one of $(\text{I})_{1}.$

  $(\text{I})_{3}$ If $\alpha_{p,q} ( T)=0$ then $T=0.$
  
  We have $\tau_q S\tau_q=0,$ this implies $S=0.$ For every $\varphi \in L^p(G)\cap L^q(G)$
  $S\varphi=\tau_pT\tau_p \varphi,$ consequently $\tau_pT\tau_p \varphi=0,$ hence $T\tau_p\varphi=0.$ Consider $r\in C_{00}(G),$ we have $T[r]=T\tau_p\tau_p[r]$, taking into  account that 
  $\tau_p[r] \in L^p(G)\cap L^q(G)$ we get $T\tau_p\tau_p[r]=0$ i.e $T[r]=0$ and finally $T=0.$
  \vskip5pt
  (II) The property 2. is verified.
  
  According to Proposition 5   $\lambda_{G}^p(\mu)\in E_{p,q}.$ Let $S\in \mathcal L(L^q(G))$ such that $S\varphi=\tau_p \lambda_{G}^p(\mu)\tau_p \varphi$ for every $\varphi \in L^p(G)\cap L^q(G),$ we have $\alpha_{p,q}(\lambda_{G}^p(\mu))=\tau_qS\tau_q.$ By Proposition 3 for every $\varphi \in L^p(G)\cap L^q(G)$ we have $\tau_p \lambda_{G}^p(\mu)\tau_p \varphi=\rho_{G}^q(\mu)\varphi.$ Consequently $S\varphi=\rho_{G}^q(\mu)\varphi$ every $\varphi \in L^p(G)\cap L^q(G),$ this implies $S=\rho_{G}^q(\mu).$ We obtain $\tau_q S \tau_q =\tau_q\rho_{G}^q(\mu) \tau_q.$ It is straightforward to verify that $\tau_q\rho_{G}^q(\mu) \tau_q=\lambda_{G}^q(\mu),$ we finally conclude that  $\alpha_{p,q}(\lambda_{G}^p(\mu))=\lambda_{G}^q(\mu).$
  \vskip5pt
  (III) The property 3. is verified.

  According to Theorem 3 $uT\in E_{p,q}.$ Let $S\in \mathcal L(L^q(G))$ such that $S\varphi=\tau_p \lambda_{G}^p(\mu)\tau_p \varphi$ for every $\varphi \in L^p(G)\cap L^q(G).$ By Theorem 3 again  
 for every $\varphi \in L^p(G)\cap L^q(G)$ $\|S\varphi\|_q \leq \|u\|_{A_p} \tribare{T}_p \|\varphi\|_q.$
  Consequently for every $\varphi\in L^q(G)$ we have $\|S\varphi\|_q \leq \|u\|_{A_p} \tribare{T}_p \|\varphi\|_q.$ It follows that  for every $\varphi\in L^q(G)$ we have $\|\tau_q S \tau_q \varphi\|_q=\|S\tau_q \varphi\|_q \leq u\|_{A_p} \tribare{T}_p \|\tau_q\varphi\|_q=\| u\|_{A_p} \tribare{T}_p \|\varphi\|_q,$ thus $\tribare{\tau_q S \tau_q}_q \leq \tribare{T}_p \|u\|_{A_p}.$ We finally have $\tribare{\alpha_{p,q}(T)}_q=\tribare{\tau_q S \tau_q}_q \leq \tribare{T}_p \|u\|_{A_p}.$
  \vskip5pt
(IV) Proof of 4.

Let $\big((k_n)_{n=1}^{\infty},(l_n)_{n=1}^{\infty}\big) \in \mathcal A_{q}(G)$ such that $k_n, l_n \in C_{00}(G)$ for every $n\geq 1$ and such that $$\sum_{n=1}^{\infty} \overline{k_n} * \check{l_n}=0.$$We have $\big<\alpha_{p,q}(T)[\tau_q k_n], [\tau_{q'} l_n]\big>_{L^q,L^{q'}}=\big<T[\tau_p k_n], [\tau_{p'} l_n]\big>_{L^p,L^{p'}}$ for every $n\geq 1.$ By  [1] Lemma 5 section 4.1 page 48
$$\sum_{n=1}^{\infty}\big<T[\tau_p k_n], [\tau_{p'} l_n]\big>_{L^p,L^{p'}} =0$$and therefore
$$\sum_{n=1}^{\infty}\big<\alpha_{p,q}(T)[\tau_{q} k_n], [\tau_{q'} l_n]\big>_{L^q,L^{q'}} =0.$$ The Corollary 8 of [1] section 4.1 page 52 implies that $\alpha_{p,q}(T)\in PM_q(G).$ 

\end{proof}
\vskip5pt
The following theorem is the main result of the paper.
\begin{theo} Let $G $ be a locally compact group, $p, q\in \R$ with $p>1, \hskip1pt p\not=2$ and $q$ between $p$ and $2$. Let $u$ be a function of $A_q(G)$ and $v\in A_p(G).$ Then:

1.  $uv\in A_p(G)$ and $\|uv\|_{A_p} \leq \|u\|_{A_q} \|u\|_{A_p},$ i.e $A_p(G)$ is a Banach module on  the Banach algebra $A_q(G),$

2.   for  every $T\in PM_p(G)$ we have
$$\big<uv, T\big>_{A_p, PM_p}=\big<u, \alpha_{p,q}(vT)\big>_{A_q, PM_q}$$(according to Theorem 7  $\alpha_{p,q}(vT) \in PM_q(G)$).

\end{theo}
\begin{proof}

(I) Definition of a linear map $\omega$ (depending of $u$) of $A_p(G)'$ into $A_p(G)''.$

There is  $\big((k_n)_{n=1}^{\infty},(l_n)_{n=1}^{\infty}\big)\in \mathcal A_{q}(G)$ such that
$$u=\sum_{n=1}^{\infty} \overline{k_n} *\check{l_n}.$$Let $F$ be an arbitrary element of $A_p(G)'.$ We put
$$\omega(F)=\sum_{n=1}^{\infty}\overline{\big<\alpha_{p,q}(v(\Psi_{G}^{p})^{-1}(F))[\tau_q k_n],[\tau_{q'} l_n]\big>}_{L^q,L^{q'}}.$$ It is easy to verify that for $\gamma \in \C$ \hskip5pt$\omega(\gamma F)=\gamma \omega(F)$ and that $\omega(F+F')=\omega(F)+\omega(F')$ for every $F, F'\in A_p(G)'.$ Moreover
$$|\omega(F)|\leq \sum_{n=1}^{\infty} \tribare{\alpha_{p,q}(v(\Psi_{G}^{p})^{-1}(F))}_q N_{q} (k_n) N_{q'}(l_n)$$according to Theorem 7 point 3. the last expression is not larger than
$$\sum_{n=1}^{\infty}\|v\|_{A_p}\tribare{(\Phi_{G}^p)^{-1}(F)}_p N_{q} (k_n) N_{q'}(l_n)$$
$$=\|v\|_{A_p} \|F\|_{A_p'} \sum_{n=1}^{\infty}N_{q} (k_n) N_{q'}(l_n) $$this implies that $\omega(F)\in A_p(G)''.$
\vskip5pt
(II) Let $F$ be an element of $A_p(G)'$ and $(F_{i})_{i\in I}$ a net of $A_p(G)'$ such that $\lim F_i=F$ for the topology $\sigma(A_p', A_p).$ We assume the existence of $C>0$ with 
$\|F_i\|_{A_{p}'} \leq C$ for every $i.$ Then $\lim \omega(F_i)=\omega(F).$

By  [1] Theorem 6  section 4.1 page 49 $\lim (\Psi_{G}^p)^{-1}(F_i) =(\Psi_{G}^p)^{-1}(F)$ for the ultraweak topology on $\mathcal L(L^p(G)).$ We also have $\tribare{(\Psi_{G}^p)^{-1}(F_i)}_p \leq C$ for every $i\in I.$

$\text{(II)}_1$  Let $(S_i)_{i\in I}$ be a net of $PM_p(G)$ and  $S\in PM_p(G).$ Suppose that $\lim S_i=S$ for the ultraweak topology on $\mathcal L(L^p(G)).$ Then for every $v\in A_p(G)$ we have $\lim vS_i=vS$ for the ultraweak topology on $\mathcal L(L^p(G)).$

For every $a\in A_p(G)$ we have $\lim <a, S_i>_{A_p, PM_p}=<a, S>_{A_p, PM_p}.$ In particular $\lim <av, S_i>_{A_p, PM_p}=<av, S>_{A_p, PM_p}$ hence $\lim <a,v S_i>_{A_p, PM_p}$
$=<a, vS>_{A_p, PM_p}$ and consequently $\lim vS_i=vS$ for the ultraweak topology on $\mathcal L(L^p(G)).$

$\text{(II)}_2$ End of the proof of (II).

By $\text{(II)}_1$ $\lim v(\Psi_{G}^p)^{-1}(F_i) =v(\Psi_{G}^p)^{-1}(F)$ for the ultraweak topology on $\mathcal L(L^p(G)).$ Moreover for every $i\in I$
$$\tribare{v(\Psi_{G}^p)^{-1}(F_i)}_p\leq \|v\|_{A_p} \tribare{(\Psi_{G}^{p})^{-1}(F_i)}_p\leq \|v\|_{A_p} C.$$
Let $r, s\in C_{00}(G).$ We have$$\lim\big<v(\Psi_{G}^p)^{-1}(F_i)[\tau_p r],[\tau_{p'}s]\big>_{L^p, L^{p'}}=\big<v(\Psi_{G}^p)^{-1}(F)[\tau_p r],[\tau_{p'}s]\big>_{L^p, L^{p'}}.$$Taking in account that
$[r]\in L^p(G) \cap L^q(G)$ and Theorem 3 we have
$$\tau_q \alpha_{p,q} \big(v(\Psi_{G}^p)^{-1}(F)\big)\tau_q[r]=\tau_pv(\Psi_{G}^p)^{-1}(F)\tau_p[r]$$and$$\tau_q \alpha_{p,q} \big(v(\Psi_{G}^p)^{-1}(F_i)\big)\tau_q[r]=\tau_pv(\Psi_{G}^p)^{-1}(F_i)\tau_p[r].$$Consequently
$$\big<v(\Psi_{G}^p)^{-1}(F)[\tau_p r],[\tau_{p'}s]\big>_{L^p, L^{p'}}
=\big<\alpha_{p,q}\big(v(\Psi_{G}^p)^{-1}(F)\big)[\tau_q r],[\tau_{q'}s]\big>_{L^q, L^{q'}}$$and
$$\big<v(\Psi_{G}^p)^{-1}(F_i)[\tau_p r],[\tau_{p'}s]\big>_{L^p, L^{p'}}
=\big<\alpha_{p,q}\big(v(\Psi_{G}^p)^{-1}(F_i)\big)[\tau_q r],[\tau_{q'}s]\big>_{L^q, L^{q'}}.$$But for every $i\in I$
$$\tribare{\alpha_{p,q}\big(v(\Psi_{G}^p)^{-1}(F_i)\big)}_q \leq \|v\|_{A_p} \tribare{(\Psi_{G}^{p})^{-1}(F_i)}_p \leq \|v\|_{A_p} C.$$It follows that
$$\lim \alpha_{p,q}\big(v(\Psi_{G}^p)^{-1}(F_i)\big)=\alpha_{p,q}\big(v(\Psi_{G}^p)^{-1}(F)\big)$$ultraweakly on $\mathcal L(L^q(G)).$ In particular
$$\lim\sum_{n=1}^{\infty} \big< \alpha_{p,q}\big(v(\Psi_{G}^p)^{-1}(F_i)\big)[\tau_q k_n], [\tau_{q'} l_n]\big>_{L^q, L^{q'}}$$
$$=\sum_{n=1}^{\infty} \big< \alpha_{p,q}\big(v(\Psi_{G}^p)^{-1}(F)\big)[\tau_q k_n], [\tau_{q'} l_n]\big>_{L^q, L^{q'}}$$ i.e\hskip5pt   $\lim \omega(F_i)= \omega(F).$ 

\vskip5pt
 (III) There is $w\in A_p(G)$ such that  $\omega(\Psi_{G}^{p}(T))=<w, T>_{A_p, PM_p}$ for every $T\in PM_p(G).$

 According to $[1] $ Theorem 6 section 4.2 page 54 there is $w\in A_p(G)$ with $\omega(F)= F(w)$ for every $F\in A_p(G)'.$ For every  $T\in PM_p(G)$ we get
 $\omega(\Psi_{G}^{p}(T))=\Psi_{G}^{p}(T)(w)=<w, T>_{A_p, PM_p}.$
 \vskip5pt
 (IV) For every $T\in PM_p(G)$ we have$$<w, T>_{A_p, PM_p}=<u, \alpha_{p,q}(vT)>_{A_q, PM_q}.$$

 We have$$<w, T>_{A_p, PM_p}=\omega(\Psi_{G}^{p}(T))$$
 $$=\sum_{n=1}^{\infty}\overline{\big<\alpha_{p,q}\big(v (\Psi_{G}^{p})^{-1}(\Psi_{G}^{p}(T))\big)[\tau_q k_n], [\tau_{q'}l_n]\big>}_{L^q, L^{q'}}$$
 $$=\sum_{n=1}^{\infty} \overline{\big<\alpha_{p,q}(vT)[\tau_{q} k_n], [\tau_{q'}l_n]\big>}_{L^q, L^{q'}}=<u, \alpha_{p,q}(vT)>_{A_q, PM_q}.$$
 \vskip5pt
(V) End of the proof of Theorem 8.

For every $\mu\in M^1(G)$ we have$$<w, \lambda_{G}^p(\mu)>_{A_p, PM_p}=<u, \alpha_{p,q}(v\lambda_{G}^p(\mu))>_{A_q, PM_q}$$but $v\lambda_{G}^p(\mu)=\lambda_{G}^{p}(\tilde v \mu)$ and $\alpha_{p,q}(\lambda_{G}^{p}(\tilde v \mu))= \lambda_{G}^{q}(\tilde v \mu).$ From $<w,\lambda_{G}^p(\mu)>_{A_p, PM_p}=\Psi_{G}^p(\lambda_{G}^p(\mu))(w)=\tilde \mu(w)$ and
$$\big<u, \lambda_{G}^{q}(\tilde v \mu)\big>_{A_q, PM_q}=\tilde \mu(uv),$$it follows  $\tilde \mu(uv)=\tilde \mu(w).$ This implies $w=uv$ and therefore $uv\in A_p(G).$ We get moreover
$$\big<uv, T\big>_{A_p, PM_p}=\big<u, \alpha_{p,q}(vT)\big>_{A_q, PM_q}$$for every $T\in PM_p(G).$ We obtain the following estimate
$$\big|\big<uv, T\big>_{A_p, PM_p}\big|
=\big|\big<u, \alpha_{p,q}(vT)\big>_{A_q, PM_q}\big|\leq \|u\|_{A_q} \tribare{\alpha_{p,q}(vT)}_q$$
$$\leq \|u\|_{A_q} \|v\|_{A_p} \tribare{T}_p.$$We finally conclude that $\|uv\|_{A_p} \leq \|u\|_{A_q} \|v\|_{A_p}.$
\end{proof}
\vskip5pt
\begin{rem} For a related result see Theorem B of [3]. Our approach is based on properties of $A_p(G)$ contained in $[1]$.
\end{rem} 
 \section{Amenable groups}
 We apply  the results of paragraph 2 to the case of amenable groups.
\begin{theo} Let $G$ be an amenable locally compact group, $p,q\in \R$ with $p>1, p\not=2$ and $q$ between $2$ and $p$. Then for every $T\in CV_p(G)$ we have $T\in E_{p,q}$ and $\tribare{\alpha_{p,q}(T)}_q \leq \tribare{T}_p.$
\end{theo}
\begin{proof}
(I) For every $\varphi\in L^p(G)\cap L^q(G)$ we have $\tau_p T \tau_p  \varphi\in L^q(G)$ and 
$\tribare{\tau_p T \tau_p  \varphi}_q \leq  \tribare{T}_p \|\varphi\|_q$.

It suffices to verify (I) for $[\varphi]$ with $\varphi\in C_{00}(G).$ Let $\psi\in C_{00}(G)$ with $N_{q'}(\psi) \leq 1$ and $\varepsilon>0.$ By $[1]$ Lemma 1 section 5.4 page 80 there is $k,l \in C_{00}(G)$ with $k\geq 0,\hskip3pt l\geq 0, \hskip3pt N_{p}(k)=N_{p'}(l)=\displaystyle \int_{G}k(t) l(t)dt=1$ and such that
$$\big|\big< \big((k*\check{l})T\big)[\tau_p \varphi],[\tau_{p'}\psi]\big>_{L^p, L^{p'}}
-\big<T[\tau_p \varphi],[\tau_{p'}\psi]\big>_{L^p, L^{p'}}\big|<\varepsilon.$$But
$$\big< \big((k*\check{l})T\big)[\tau_p \varphi],[\tau_{p'}\psi]\big>_{L^p, L^{p'}}
=\big<\tau_{p} \big((k*\check{l})T\big)\tau_p[ \varphi],[\psi]\big>_{L^p, L^{p'}}$$and according 
to Theorem 3 $\tau_{p} \big((k*\check{l})T\big)\tau_p[ \varphi]$ belongs to $L^q(G)$ this implies
$$\big<\tau_{p} \big((k*\check{l})T\big)\tau_p[ \varphi],[\psi]\big>_{L^p, L^{p'}}
=\big<\tau_{p} \big((k*\check{l})T\big)\tau_p[ \varphi],[\psi]\big>_{L^q, L^{q'}}$$and therefore
$$\big|\big< \big((k*\check{l})T\big)[\tau_p \varphi],[\tau_{p'}\psi]\big>_{L^p, L^{p'}}\big|
\leq \|\tau_{p} \big((k*\check{l})T\big)\tau_p[ \varphi]\|_q \|[\psi]\|_{q'}.$$Using again Theorem 3 we have
$$\|\tau_{p} \big((k*\check{l})T\big)\tau_p[ \varphi]\|_q 
\leq \|k*\check{l}\|_{A_p} \tribare{T}_p \|[\varphi]\|_q \leq  \tribare{T}_p \|[\varphi]\|_q,$$and obtain
$$\big|\big< \big((k*\check{l})T\big)[\tau_p \varphi],[\tau_{p'}\psi]\big>_{L^p, L^{p'}}\big|
\leq \tribare{T}_p \|[\varphi]\|_q \|\psi\|_{q'}\leq \tribare{T}_p \|[\varphi]\|_q$$and therefore
$$\big|\big<T[\tau_p \varphi].[\tau_{p'}\psi]\big>_{L^p, L^{p'}}\big|
<\varepsilon+ \tribare{T}_p \|[\varphi]\|_q.$$Thus
$$\big|\big<T[\tau_p \varphi].[\tau_{p'}\psi]\big>_{L^p, L^{p'}}
\big|\leq  \tribare{T}_p \|[\varphi]\|_q,$$$$ \big|\big<\tau_pT \tau_p[ \varphi].[\psi]\big>_{L^p, L^{p'}}
\big|\leq  \tribare{T}_p \|[\varphi]\|_q$$and finally $\tau_p T\tau_p [\varphi] \in L^q(G)$ with 
$\|\tau_p T\tau_p [\varphi]\|_q \leq \tribare{T}_p \|[\varphi]\|_q.$
\vskip5pt

(II) $\tribare{\alpha_{p,q}(T)}_q \leq \tribare{T}_p$.

There is a unique $S\in \mathcal L(L^q(G))$ such that $S\varphi=\tau_p T\tau_p \varphi$ for every $\varphi\in L^p(G) \cap L^q(G).$ By (I) $\tribare{S}_q \leq \tribare{T}_p,$ consequently 
$\tribare{\alpha_{p,q}(T)}_q=\tribare{\tau_q S \tau_q}_q=\tribare{S}_q\leq \tribare{T}_p.$
\end{proof}
\setcounter{cor}{9}
\begin{cor} Let $G$ be an amenable locally compact group, $p,q\in \R$ with $p>1, p\not=2$ and $q$ between $2$ and $p$. Then:

1.  $\alpha_{p,q}$ is a contractive Banach algebra monomorphism of $CV_p(G)$ into $CV_q(G),$

2. for every $\mu\in M^1(G)$ we have $\tribare{\lambda_{G}^{q}(\mu)}_q \leq \tribare{\lambda_{G}^{p}(\mu)}_p.$
\end{cor}
\begin{proof} It suffices to verify 2. . By  Theorem 9 $$\tribare{\alpha_{p,q}(\lambda_{G}^{p}(\mu))}_q \leq \tribare{\lambda_{G}^{p}(\mu)}_p$$but by Theorem 7  $\alpha_{p,q}(\lambda_{G}^{p}(\mu))=\lambda_{G}^{q}(\mu).$
\end{proof}
\begin{rem}
See [4] Corollary p. 512. For another approach for the unimodular case see [2].
\end{rem} 
\vskip5pt
We obtain  other properties of $\alpha_{p,q}$ in the next theorem.
\setcounter{theo}{10}
\begin{theo} Let $G$ be an amenable locally compact group, $p,q\in \R$ with $p>1, p\not=2$ and $q$ between $2$ and $p$. Then:

1. for every $u\in A_q(G)$ we have $u\in A_p(G)$ and $\|u\|_{A_p}\leq \|u\|_{A_q},$

2.  for every $u\in A_q(G)$ and for every $T\in PM_p(G)$ we have
$$\big<u,T\big>_{A_p, PM_p}= \big< u, \alpha_{p,q}(T)\big>_{A_q, PM_q},$$

3. for every $u\in A_q(G)$ and for every $T\in CV_p(G)$ we have
$$\alpha_{p,q}(uT)=u \hskip1pt\alpha_{p,q}(T).$$
\end{theo}
\begin{proof} Let $u$ be a function of $A_q(G).$
\vskip5pt

(I) Definition of a linear map $\omega$ (depending  of $u$) of $A_p(G)'$ into $A_p(G)''.$

There is $\big((k_n)_{n=1}^{\infty},(l_n)_{n=1}^{\infty}\big)\in \mathcal A_{q}(G)$ such that
$$u=\sum_{n=1}^{\infty} \overline{k_n}*\check{l_n}.$$For every $F\in A_p(G)'$ we put
$$\omega(F)=\sum_{n=1}^{\infty} \overline{\big<\alpha_{p,q}\big((\Psi_{G}^{p})^{-1}(F)\big)[\tau_{q} k_n], [\tau_{q'} l_n]\big>}_{L^q, L^{q'}}.$$For every $\gamma \in \C$ and $F, F'\in A_p(G)'$ we have
$\omega(\gamma F)=\gamma\hskip1pt \omega(F)$ and $\omega(F+F')=\omega(F)+\omega(F').$ We now show that $\omega(F)\in A_p(G)''.$ Observe at first that
$$|\omega(F)|
\leq \sum_{n=1}^{\infty} \big|\big<\alpha_{p,q}\big((\Psi_{G}^{p})^{-1}(F)\big)[\tau_q k_n], [\tau_{q'} l_n]\big>_{L^q, L^{q'}}\big|,$$taking in account Theorem 9 we have for every $n\geq 1$
$$\big|\big<\alpha_{p,q}\big((\Psi_{G}^{p})^{-1}(F)\big)[\tau_q k_n], [\tau_{q'} l_n]\big>_{L^q, L^{q'}}\big|$$
$$\leq \tribare{(\Psi_{G}^{p})^{-1}(F)}_p N_q(k_n) N_{q'}(l_n)=\|F\|_{A_p'} N_q(k_n) N_{q'}(l_n)$$hence
$$|\omega(F)|\leq \|F\|_{A_p'} \sum_{n=1}^{\infty} N_{q}(k_n) N_{q'}(l_n).$$This implies that  $\omega(F)\in A_p(G)''.$
\vskip5pt
(II) Let $F$ be an element of $A_p(G)'$ and $(F_i)_{i\in I}$ a net of $A_p(G)'$ such that $\lim F_i=F$
for the topology $\sigma(A_p', A_p'').$ We assume the existence of $C>0$ such that $\|F_i\|_{A_p'} \leq C$ for every $i\in I.$ Then $\lim \omega(F_i)=\omega(F).$

We have $\lim (\Psi_{G}^{p})^{-1}(F_i)=(\Psi_{G}^{p})^{-1}(F)$ ultraweakly in $\mathcal L(L^p(G)).$ We also have $\tribare{(\Psi_{G}^{p})^{-1}(F_i)}_p \leq C$ for every $i\in I.$

Let  $r, s\in C_{00}(G).$ We have
$$\lim\big <(\Psi_{G}^{p})^{-1}(F_i)[\tau_p r], [\tau_{p'} s]\big>_{L^p, L^{p'}}=\big <(\Psi_{G}^{p})^{-1}(F)[\tau_p r], [\tau_{p'} s]\big>_{L^p, L^{p'}}.$$Taking in account that $[r]\in L^p(G)\cap L^q(G)$
and that $(\Psi_{G}^{p})^{-1}(F_i), (\Psi_{G}^{p})^{-1}(F_i) \in E_{p,q} $ we get
 $$\tau_q  \alpha_{p,q}\big((\Psi_{G}^{p})^{-1}(F_i)\big)\tau_q [r]=\tau_p(\Psi_{G}^{p})^{-1}(F_i) \tau_p [r]$$and$$\tau_q  \alpha_{p,q}\big((\Psi_{G}^{p})^{-1}(F)\big)\tau_q [r]=\tau_p(\Psi_{G}^{p})^{-1}(F) \tau_p [r].$$This implies
 $$\big <(\Psi_{G}^{p})^{-1}(F)[\tau_p r], [\tau_{p'} s]\big>_{L^p, L^{p'}}
 =\big<\tau_q\alpha_{p,q}\big((\Psi_{G}^{p})^{-1}(F)\big)\tau_q[r], [s]\big>_{L^q, L^{q'}}$$and for every $i\in I$
 $$\big <(\Psi_{G}^{p})^{-1}(F_i)[\tau_p r], [\tau_{p'} s]\big>_{L^p, L^{p'}}
 =\big<\tau_q\alpha_{p,q}\big((\Psi_{G}^{p})^{-1}(F_i)\big)\tau_q[r], [s]\big>_{L^q, L^{q'}}.$$We have therefore
 $$\lim\big<\alpha_{p,q}\big((\Psi_{G}^{p})^{-1}(F_i)\big)[\tau_q r], [\tau_{q'} s]\big>_{L^q, L^{q'}}$$
 $$= \big<\alpha_{p,q}\big((\Psi_{G}^{p})^{-1}(F)\big)[\tau_q r], [\tau_{q'} s]\big>_{L^q, L^{q'}}.$$But for every $i\in I$
 $$\tribare{\alpha_{p,q}\big((\Psi_{G}^{p})^{-1}(F_i)\big)}_q \leq \tribare{(\Psi_{G}^{p})^{-1}(F_i)}_p
 \leq C,$$we obtain that
 $$\lim \alpha_{p,q}\big((\Psi_{G}^{p})^{-1}(F_i)\big)=\alpha_{p,q}\big((\Psi_{G}^{p})^{-1}(F_i)\big)$$ultraweakly on $\mathcal L(L^q(G)).$ This finally implies that
 $$\lim\sum_{n=1}^{\infty}\big<\alpha_{p,q}\big((\Psi_{G}^{p})^{-1}(F_i)\big)[\tau_q k_n], [\tau_{q'} l_n]\big>_{L^q, L^{q'}}$$
 $$=\sum_{n=1}^{\infty}\big<\alpha_{p,q}\big((\Psi_{G}^{p})^{-1}(F)\big)[\tau_q k_n], [\tau_{q'}  l_n]\big>_{L^q, L^{q'}}$$i.e. $\lim \omega(F_i)=\omega(F).$
 \vskip5pt
 (III) End of the proof of Theorem 11.

 As in the proof of Theorem 8 (see step (III)) there is $v\in A_p(G)$ such that $\omega(F)= F(v)$ for every $F\in A_p(G)'.$ For every $T\in CV_p(G)$ we have
 $$\omega(\Psi_{G}^p(T))=\Psi_{G}^p(T)(v)=\big<v, T\big>_{A_p, PM_p}$$
$$=\sum_{n=1}^{\infty} \overline{\big<\alpha_{p,q}\big((\Psi_{G}^p)^{-1}(\Psi_{G}^p(T))\big)[\tau_q k_n], [\tau_{q'} l_n]\big>}_{L^q, L^{q'}}$$
$$=\sum_{n=1}^{\infty} \overline{\big<\alpha_{p,q}(T)[\tau_q k_n], [\tau_{q'} l_n]\big>}_{L^q, L^{q'}}
=\big<u,\alpha_{p,q}(T)\big>_{A_q, PM_q}.$$In particular for every $\mu\in M^1(G)$ we have
$$\big<v, \lambda_{G}^p(\mu)\big>_{A_p, PM_p}=\big<u,\alpha_{p,q}(\lambda_{G}^p(\mu))\big>_{A_q, PM_q}=\big<u, \lambda_{G}^q(\mu)\big>_{A_q, PM_q}.$$This implies $\tilde{\mu}(v)=\tilde{\mu}(u)$ and therefore $u=v,$ we obtain that $u\in A_p(G).$ We also obtain for every $u\in A_q(G)$ and for every $ T\in CV_p(G)$
$$\big<u, T\big>_{A_p, PM_p} =\big<u, \alpha_{p,q}(T)\big>_{A_q, PM_q}. $$This implies the following estimate:
$$\big|\big<u, T\big>_{A_p, PM_p}\big|
 =\big|\big<u, \alpha_{p,q}(T)\big>_{A_q, PM_q}\big|\leq \|u\|_{A_q} \tribare{\alpha_{p,q}(T)}_q$$
 $$\leq \|u\|_{A_q} \tribare{T}_p$$thus $\|u\|_{A_p} \leq \|u\|_{A_q}.$

 It remains to verify that for every $u\in A_q(G)$ and for every $T\in CV_p(G)$ we have
 $$\alpha_{p,q}(uT)=u\hskip1pt \alpha_{p,q}(T).$$Consider an arbitrary $v\in A_q(G).$ We have
 $$\big<v, \alpha_{p,q}(uT)\big>_{A_q, PM_q}=\big<v, uT \big>_{A_p, PM_p}.$$But
 $$\big<v, uT\big>_{A_p, PM_p}=\big<v u, T\big>_{A_p, PM_p}=\big< vu, \alpha_{p,q}(T)\big>_{A_q, PM_q}$$
 $$=\big< v, u\hskip1pt\alpha_{p,q}(T)\big>_{A_q, PM_q}.$$We get
 $$\big<v, \alpha_{p,q}(uT)\big>_{A_q, PM_q}=\big< v, u\hskip1pt\alpha_{p,q}(T)\big>_{A_q, PM_q}$$and finally
 $$\alpha_{p,q}(uT)=u\hskip1pt \alpha_{p,q}(T).$$
 \end{proof}

\section*{References}

[1] A. Derighetti, Convolution Operators on Groups, Lecture Notes of the Unione Matematica Italiana 11, Springer-Verlag Berlin Heidelberg 2011.
\vskip5pt
[2] A. Derighetti, Survey on the Fig$\grave{\hbox{a}}$\hskip2pt-Talamanca  Herz algebra, $\Sigma $ Mathematics {\bf 2019}, 7, 660; doi:10.3390/math7080660.
\vskip5pt
[3] C. Herz, The theory of $p$\hskip1pt-spaces with an application to convolution operators,  Trans. Am. Math. {\bf 154} (1971), 69\hskip1pt-82. 
\vskip5pt
[4]  C. Herz and N. Rivière, Estimates for translation\hskip1pt-invariant operators on spaces with mixed norms, Studia Mathematica {\bf XLIV} (1972), 511\hskip1pt-515.

\end{document}